\documentclass[11pt,reqno]{amsart}

\usepackage{amssymb, latexsym, amsmath}
\usepackage{amsfonts}
\usepackage{amsthm}

\theoremstyle{definition}
\newtheorem{theorem}{Theorem}

\newtheorem{definition}[theorem]{Definition}
\newtheorem{corollary}[theorem]{Corollary}

\newtheorem*{example}{Example}

\newdimen\unit\newdimen\psep\newcount\nd\newcount\ndx\newbox\dotb\newbox\ptbox
\newdimen\dx\newdimen\dy\newdimen\dxx\newdimen\dyy\newdimen\hgt
\newdimen\dxl\newdimen\dyl\newdimen\xoff\newdimen\yoff
\newcommand\clap[1]{\hbox to 0pt{\hss{#1}\hss}}
\newcommand\vdisk[1]{{\font\dotf=cmr10 scaled #1\dotf.}}
\newcommand\varline[2]{\setbox\dotb\hbox{\vdisk{#1}}\xoff=-.5\wd\dotb
\wd\dotb=0pt\yoff=-.5\ht\dotb\psep=#2\ht\dotb}
\newcommand\varpt[1]{\setbox\ptbox\clap{\vdisk{#1}}\setbox\ptbox
\hbox{\raise-.5\ht\ptbox\box\ptbox}}
\newcommand\cpt{\copy\ptbox}
\newcommand\xpt{\raise-3pt\clap{x}}
\newcommand\point[3]{\rlap{\kern#1\unit\raise#2\unit\hbox{#3}}}
\newcommand\setnd[4]{\dx=#3\unit\advance\dx-#1\unit\divide\dx by\psep
\dy=#4\unit\advance\dy-#2\unit\divide\dy by\psep
\multiply\dx by\dx\multiply\dy by\dy\advance\dx\dy\nd=1\advance\dx-1sp
\loop\ifnum\dx>0\advance\dx-\nd sp\advance\nd1\advance\dx-\nd sp\repeat}
\newcommand\dl[4]{{\setnd{#1}{#2}{#3}{#4}\dline{#1}{#2}{#3}{#4}\nd}}
\newcommand\dline[5]{{\nd=#5\hgt=#2\unit\dx=#3\unit\advance\dx-#1\unit
\divide\dx by\nd\dy=#4\unit\advance\dy-#2\unit\divide\dy by\nd
\advance\hgt\yoff\rlap{\kern#1\unit\kern\xoff\loop\ifnum\nd>1\advance\nd-1
\advance\hgt\dy\kern\dx\raise\hgt\copy\dotb\repeat}}}

\newcommand\qellip[4]{{\setnd{0}{0}{#3}{#4}\dx=\unit\dy=0pt\raise\yoff\rlap{%
\kern#1\unit\kern\xoff\raise#2\unit\hbox{\loop\ifnum\dx>0\rlap{\kern#3\dx
\raise#4\dy\copy\dotb}\hgt=\dx\divide\hgt by\nd\advance\dy\hgt\hgt=\dy
\divide\hgt by\nd\advance\dx-\hgt\repeat\rlap{\raise#4\dy\copy\dotb}}}}}
\newcommand\xellip[6]{{\setnd{0}{0}{#3}{#4}\dx=\unit\dy=0pt
\dxl=#5\unit\dyl=#6\unit\raise\yoff\rlap{%
\kern#1\unit\kern\xoff\raise#2\unit\hbox{\loop\ifnum\dx>\dxl
\ifnum\dy>\dyl\rlap{\kern#3\dx\raise#4\dy\copy\dotb}\fi
\hgt=\dx\divide\hgt by\nd\advance\dy\hgt\hgt=\dy
\divide\hgt by\nd\advance\dx-\hgt\repeat}}}}

\newcommand\bez[6]{{\setnd{#1}{#2}{#3}{#4}\ndx=\nd\setnd{#3}{#4}{#5}{#6}
\ifnum\ndx>\nd\nd=\ndx\fi\dx=#3\unit\advance\dx-#1\unit\dy=#4\unit
\advance\dy-#2\unit\dxx=#5\unit\advance\dxx-#1\unit\dyy=#6\unit\advance
\dyy-#2\unit\advance\dxx-2\dx\advance\dyy-2\dy\divide\dxx by\nd\divide\dyy
by\nd\advance\dx.25\dxx\advance\dy.25\dyy\divide\dx by\nd\divide\dy by\nd
\multiply\nd by2\dx=100\dx\dy=100\dy\dxx=100\dxx\dyy=100\dyy\divide\dxx by\nd
\divide\dyy by\nd\hgt=#2\unit\raise\yoff\rlap{\kern#1\unit\kern\xoff
\raise\hgt\copy\dotb\loop\ifnum\nd>0\advance\nd-1\advance\hgt0.01\dy
\kern0.01\dx\raise\hgt\copy\dotb\advance\dx\dxx\advance\dy\dyy\repeat}}}
\newcommand\ptu[3]{\point{#1}{#2}{\cpt\raise1ex\clap{$\scriptstyle{#3}$}}}
\newcommand\ptd[3]{\point{#1}{#2}{\cpt\raise-1.8ex\clap{$\scriptstyle{#3}$}}}
\newcommand\ptr[3]{\point{#1}{#2}{\cpt\raise-.4ex\rlap{$\ \scriptstyle{#3}$}}}
\newcommand\ptl[3]{\point{#1}{#2}{\cpt\raise-.4ex\llap{$\scriptstyle{#3}\ $}}}
\newcommand\pxu[3]{\point{#1}{#2}{\xpt\raise1ex\clap{$\scriptstyle{#3}$}}}
\newcommand\pxd[3]{\point{#1}{#2}{\xpt\raise-1.8ex\clap{$\scriptstyle{#3}$}}}
\newcommand\pxr[3]{\point{#1}{#2}{\xpt\raise-.4ex\rlap{$\ \scriptstyle{#3}$}}}
\newcommand\pxl[3]{\point{#1}{#2}{\xpt\raise-.4ex\llap{$\scriptstyle{#3}\ $}}}
\newcommand\ptlu[3]{\point{#1}{#2}{\raise.8ex\clap{$\scriptstyle{#3}$}}}
\newcommand\ptld[3]{\point{#1}{#2}{\raise-1.6ex\clap{$\scriptstyle{#3}$}}}
\newcommand\ptlr[3]{\point{#1}{#2}{\raise-.4ex\rlap{$\,\scriptstyle{#3}$}}}
\newcommand\ptll[3]{\point{#1}{#2}{\raise-.4ex\llap{$\scriptstyle{#3}\,$}}}

\newcommand\thnline{\varline{400}{.4}}
\newcommand\dotline{\varline{800}{2.5}}
\varpt{2500}\thnline\unit=2em
\newcommand\fan[2]{\point{#1}{#2}{$\thnline\dl00{.5}{.2}\dl00{.5}{-.2}$}}
\newcommand\naf[2]{\point{#1}{#2}{$\thnline\dl00{-.5}{.2}\dl00{-.5}{-.2}$}}

\title{Limited packings of closed neighbourhoods in graphs}

\author[P. N. Balister]{Paul N. Balister}
\address{Department of Mathematical Sciences, University of Memphis, Memphis TN 38152, USA}
\email{pbalistr@memphis.edu}

\author[B. Bollob\'{a}s]{B\'{e}la Bollob\'{a}s}
\address{Department of Pure Mathematics and Mathematical Statistics, University of Cambridge, Wilberforce Road, Cambridge CB3\thinspace0WB, UK; {\em and\/}
Department of Mathematical Sciences, University of Memphis, Memphis TN 38152, USA; {\em and\/} London Institute for Mathematical Sciences, 35a South St., Mayfair, London W1K\thinspace2XF, UK.}
\email{b.bollobas@dpmms.cam.ac.uk}

\author[K. Gunderson]{Karen Gunderson}
\address{Heilbronn Institute for Mathematical Research, School of Mathematics, University of Bristol, Bristol BS8 1TW, UK.}
\email{karen.gunderson@bristol.ac.uk}

\date{14 Dec 2014}


\begin{document}

\begin{abstract}
The k-limited packing number, $L_k(G)$, of a graph $G$, introduced by Gallant, Gunther, Hartnell, and Rall, is the maximum cardinality of a set $X$ of vertices of $G$ such that every vertex of $G$ has at most $k$ elements of $X$ in its closed neighbourhood. The main aim in this paper is to prove the best-possible result that if $G$ is a cubic graph, then $L_2(G) \geq |V (G)|/3$, improving the previous lower bound given by Gallant, \emph{et al.}

In addition, we construct an infinite family of graphs to show that lower bounds given by Gagarin and Zverovich are asymptotically best-possible, up to a constant factor, when $k$ is fixed and $\Delta(G)$ tends to infinity. For $\Delta(G)$ tending to infinity and $k$ tending to infinity sufficiently quickly, we give an asymptotically best-possible lower bound for $L_k(G)$, improving previous bounds.
\end{abstract}

\maketitle

\section{Introduction}\label{sec:intro}

Limited packings in graphs were introduced by Gallant, Gunther, Hartnell and Rall \cite{GGHR10} as a generalization of certain types of neighbourhood packings.  For a graph $G$ and vertex $v$, let $N[v] = \{v\} \cup N(v)$ denoted the \emph{closed neighbourhood of $v$}.  For $k \geq 1$ and a graph $G$, a set $X \subseteq V(G)$ is called a \emph{$k$-limited packing} if for every $v \in V(G)$, $|N[v] \cap X| \leq k$.  

In the case $k = 1$, a $1$-limited packing is precisely a set of vertices where every pair is at distance at least $3$. Meir and Moon \cite{MM75} defined a distance $k$-packing in a graph to be a set of vertices $X$ with the property that for every $x,y \in X$, $d(x,y) > k$. Thus, a distance $2$-packing is the same as a $1$- limited packing. For $k > 1$, there is no direct connection between $k$-limited packings and distance packings.

The notion of a $k$-limited packing can be rephrased in terms of subsets of hypergraphs.  Given a graph $G$, define a hypergraph with vertices $V(G)$ and whose hyperedges are the closed neighbourhoods of vertices in $G$.  A $k$-limited packing in $G$ corresponds to a subset of the hypergraph with maximum vertex degree $k$.

The central question examined here is the order of the largest $k$-limited packing in a graph.

\begin{definition}
For a graph $G$ and $k \geq 1$, the \emph{$k$-limited packing number of $G$} is
\[
L_k(G) = \max\{|X| : X \subseteq V(G) \text{ is a $k$-limited packing}\}.
\]
\end{definition} 

Note that if $k > \Delta(G)$, then $L_k(G) = |V(G)|$ and for a fixed graph $G$, the function $L_k(G)$ is non-decreasing in $k$.

For $2$-regular graphs, $k$-limited packing numbers can be determined exactly. For any $n \geq 3$, the cycle $C_n$ has $1$-limited packing number $L_1(C_n) = \lfloor n/3\rfloor$ and $2$-limited packing number $L_2(C_n) = \lfloor 2n/3 \rfloor$, as was noted in \cite{GGHR10}. Considering the fraction of vertices in a limited packing, one can see that for any $2$-regular graph $G$, $L_1(G) \geq n/5$, which is achieved by graphs whose connected components are all copies of $C_5$. Further, if $G$ is a $2$-regular graph, then $L_2(G) \geq n/2$, which is achieved by graphs whose connected components are copies of $C_4$.

For any graph $G$ with $\Delta(G) = \Delta$, the set of vertices at distance at most $2$ from any particular vertex of $G$ has at most $\Delta^2 + 1$ vertices. Thus, by a greedy choice of a $1$-limited set, we have that
\[
L_1(G) \geq \frac{ |V(G)|}{\Delta^2 + 1}.
\]
In particular, for any cubic graph $G$, the $1$-limited packing number is at least $L_1 (G) \geq |V (G)|/10$ . This lower bound is achieved by the Petersen graph which has $10$ vertices and diameter $2$. By taking graphs consisting of many vertex-disjoint copies of the Petersen graph, this shows that for every $n$ divisible by $10$, there is a graph on $n$ vertices with $L_1(G) = |V (G)|/10$.

Gallant, Gunther, Hartnell and Rall \cite{GGHR10} determined $L_k(G)$ precisely for certain classes of graphs and claim that for cubic graphs,
\begin{equation}\label{eq:cubic-bd}
\frac{1}{4}|V (G)| \leq  L_2(G) \leq \frac{1}{2}|V (G)|.
\end{equation}
In addition, they relate limited packing numbers to certain domination parameters in graphs. For a graph $G$ and $\ell \geq 1$, a set $D\subseteq V(G)$ is called an \emph{$\ell$-tuple dominating set} if for every $v \in V(G)$, $|N[v]\cap D| \geq \ell$. For $\ell=1$, a $1$-tuple dominating set is a dominating set in the usual sense. If $G$ is $r$-regular, then a set $D$ is an $\ell$-tuple dominating set if{f} $V (G) \setminus D$ is a $(r+1-\ell)$-limited packing. Thus, bounds on the limited packing numbers of regular graphs can be interpreted as bounds on multiple domination.

In \cite{GZ13}, Gagarin and Zverovich use a random approach to show that for a graph $G$ with $\Delta(G) = \Delta$ and $k \leq \Delta(G)$,
\begin{equation}\label{eq:random-del-bd}
L_k(G) \geq n \frac{k}{(k+1) \sqrt[k]{\binom{\Delta}{k}(\Delta+1)}}.
\end{equation}

In Section \ref{sec:cubic}, we show that if $G$ is graph with $\Delta(G) \leq 3$, then $L_2(G) \geq |V (G)|/3$, with an example showing this is best-possible. This improves the previous best-known lower bound of $L_2(G) \geq |V (G)|/4$, as in equation \eqref{eq:cubic-bd}. In the case $k=2$ and $\Delta=3$, the bound given by equation \eqref{eq:random-del-bd} is $L_2(G) \geq \frac{|V(G)|}{3\sqrt{3}}$ which is a worse bound than $L_2(G) \geq |V (G)|/4$. Our new lower bound on $L_2(G)$ for cubic graphs is also translated into a new and best-possible upper-bound for the size of $2$-tuple dominating sets in cubic graphs.

%
%

In Section \ref{sec:large-deg}, we give examples of graphs constructed using finite projective planes to show that for $k$ fixed and $\Delta$ tending to infinity, the lower bound from \eqref{eq:random-del-bd} is tight up to constants.  Using the Lov\'{a}sz Local Lemma, we give an improved lower bound on $L_k(G)$ for $k$ and $\Delta(G)$ tending to infinity, when $k$ grows sufficiently quickly.

\section{Cubic graphs}\label{sec:cubic}

In this section, we give tight bounds on the fraction of vertices in any $1$-limited or $2$-limited packing in any graph with maximum degree $3$.

As was noted in the introduction, a greedy choice of a $1$-limited packing shows that if $G$ is a graph with $\Delta(G) = 3$, then $L_1(G) \geq |V(G)|/10$, a lower bound which is obtained by taking graphs consisting of a union of vertex-disjoint copies of the Petersen graph.

In the case $k = 3$, the $3$-limited packing number of a $3$-regular graph $G$ is precisely dual to the usual domination number, $\gamma(G)$. As was noted in the previous section, if $G$ is a $3$-regular graph, then $L_3(G) + \gamma(G) = n$. Kostochka and Stocker \cite{KS09} showed that if $G$ is a connected cubic graph, then $\gamma(G) \leq 5|V (G)|/14$. Further, they gave examples where this is sharp, including a connected cubic graph on $14$ vertices with domination number $5$. Thus, an immediate consequence of this result is that if $G$ is a cubic graph, then $L_3(G) \geq 9|V (G)|/14$, which is tight.

Consider now $2$-limited packings in cubic graphs. In Theorem \ref{thm:cubic-bd} below, we show that every graph with maximum degree $3$ has a $2$-limited packing containing at least a third of the vertices. To see that this bound is tight, let $H_6$ be the graph on $6$ vertices consisting of a $6$-cycle with all three chords of length $3$ added.  Any $2$-limited packing of $H_6$ contains at most $2$ vertices of $H_6$. By taking multiple vertex-disjoint copies of $H_6$, one has an infinite collection of cubic graphs $G$ with $L_2(G) = |V (G)|/3$.

The proof of Theorem \ref{thm:cubic-bd} uses induction on the number of vertices, with two types of edges to keep track of additional conditions imposed on limited packings of subgraphs.

Throughout the proof of Theorem \ref{thm:cubic-bd} to come, let $G$ be a multigraph of maximum degree 3 with edges of two possible types: {\em colour edges} (c-edges)
and {\em domination edges} (d-edges).
In such a multigraph, define a set $X\subseteq V(G)$ to be a {\em$2$-limited\/} set if for any c-edge $uv$,
$|X\cap\{u,v\}|\le 1$ and for any vertex $v$, $|X\cap N_d[v]|\le 2$, where $N_d[v]=\{v\}\cup\{u:uv$
is a d-edge$\}$ is the closed d-neighbourhood of~$v$. Note that the c-edges do not contribute to
$N_d[v]$ and $X$ may contains pairs of vertices joined by d-edges. Note also that both types
of edge contribute to the degree of a vertex, and this degree must be at most~3.

Duplicate c-edges or duplicate d-edges between vertices $u$ and $v$ may be removed without changing
the conditions on $X$, however there may be both a c-edge and a d-edge joining the same
pair of vertices, so $G$ is a multigraph.

\begin{theorem}\label{thm:cubic-bd}
 If $G$ does not contain a component $K_4$ consisting entirely of c-edges,
 then there is a $2$-limited set $X$ with $|X|\ge|V(G)|/3$.
\end{theorem}

\begin{proof}
We use induction on the number of vertices. Clearly we may assume $G$ is
connected as otherwise we can apply induction to each component and take the
union of the corresponding $2$-limited sets. If $G$ has three or fewer vertices then
we can take $X$ to be any single vertex. If $G$ has 4 vertices and is not a $K_4$
with all c-edges, then pick $uv$ which is not a c-edge (possibly not an edge at all)
and let $X=\{u,v\}$. If $|V(G)|>4$ and $G$ consists only of c-edges then we
can 3-colour $G$ by Brooks' theorem. At least one colour class has at least $|V(G)|/3$
vertices and we can use this colour class for~$X$. Hence we may assume $G$ contains
at least one d-edge.

In the following we shall modify $G$ be removing vertices and occasionally
adding c-edges between vertices. The main problem is that we must avoid generating
a $K_4$ component using just c-edges. The following reduction is therefore useful.

Suppose $G$ contains the following configuration
\[
 \ptd00{c}\ptu02{a}\pxd20{d}\pxu22{b}\ptd31{u}\ptd51{v}\fan51
 \dotline\dl0002\dl0020\dl0022\dl0220\dl0222
 \thnline\dl2031\dl2231\dl3151
 \point{0}{-1}{Configuration $A$.}\hskip5\unit
\]
where dotted lines indicate c-edges and solid lines are either c-edges or d-edges
(or both when the degree condition allows it).
Pick a largest $2$-limited set $X$ for $G\setminus\{a,b,c,d,u,v\}$. Then $X\cup\{b,d\}$
is $2$-limited for $G$. As we have removed 6 vertices from $G$, $|X|\ge |V(G)|/3-2$
and so we have the required $2$-limited set for $G$. Hence we may assume $G$
does not contain configuration~$A$.

Now suppose $G$ contains a vertex $u$ adjacent to only one other vertex $v$
(possibly by both c- and d-edges). Consider the following transformation.
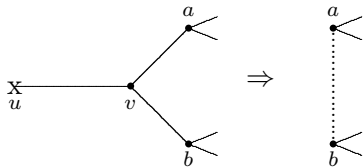
\[
 \pxd01{u}\ptd21{v}\ptd30{b}\ptu32{a}\fan30\fan32
 \thnline\dl0121\dl2130\dl2132\hskip3\unit
 \qquad\raise1\unit\hbox{$\Rightarrow$}\qquad
 \ptd00{b}\ptu02{a}\fan00\fan02\dotline\dl0002\hskip1\unit
\]
where we remove $\{u,v\}$ from $G$ and add the c-edge $ab$. The $2$-limited set is $X\cup\{u\}$,
where $X$ is a maximum $2$-limited set for the resulting graph. Note that we can assume
that adding the c-edge $ab$ will not complete a $K_4$ in c-edges by the absence
of configuration $A$. Note also that if $v$ is adjacent to fewer than 3 vertices,
or if $ab$ is already a c-edge,
or indeed, if any of $vu$, $va$, or $vb$ is not a d-edge then the same construction works
without the need to add $ab$ as a c-edge.

Now suppose that $G$ contains a vertex $u$ adjacent to only two other vertices, $v$ and $w$.
Consider the following transformation.
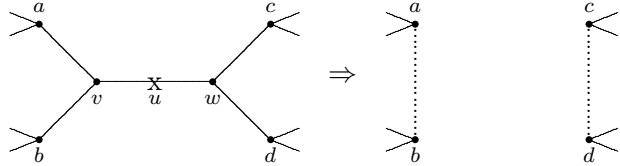
\[
 \naf00\naf02\ptd00{b}\ptu02{a}\ptd11{v}\pxd21{u}\ptd31{w}\ptd40{d}\ptu42{c}\fan40\fan42
 \thnline\dl0011\dl0211\dl1131\dl3140\dl3142\hskip4\unit
 \qquad\raise1\unit\hbox{$\Rightarrow$}\qquad
 \naf00\naf02\ptd00{b}\ptu02{a}\ptd30{d}\ptu32{c}\fan30\fan32
 \dotline\dl0002\dl3032\hskip3\unit
\]
Then $X\cup\{u\}$ is $2$-limited where $X$ is a $2$-limited set in $G\setminus\{v,u,w\}$
with the c-edges $ab$ and $cd$ added. As before there are a number of degenerate cases
where in fact either or both $ab$ and $cd$ do not need to be added. (One particular degenerate
case is when $vw$ is an edge so that $a=w$ and $c=v$ are also removed.) There is however
one special subcase that needs to be dealt with in a different way. Due to the absence
of configuration~$A$ we may assume that neither $ab$
or $cd$ generates a $K_4$ in c-edges {\em individually}.
However it is possible that $ab$ and $cd$ together form two edges of a single $K_4$ in c-edges.
However, in this case $G$ consists of just 7 vertices and we can take $\{a,b,w\}$
as our $2$-limited set. (Of course $ab$ is not an edge of $G$, $w$ is incident only
to d-edges as otherwise we would not have needed to add $cd$,
and $\{c,d\}\ne\{a,b\}$ so $w$ has at most one neighbour in $\{a,b\}$.)

Note that we can now assume $G$ contains no multiple edges (as otherwise some vertex
would be joined to at most two other vertices). We may also assume $G$ is 3-regular.
As mentioned above we may also assume it has at least one d-edge.

Suppose there exists a d-edge $uv$ that lies in two triangles. Consider the transformation
which removes all the following vertices.
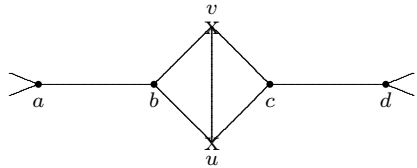
\[
 \naf01\ptd01{a}\ptd21{b}\pxd30{u}\pxu32{v}\ptd41{c}\ptd61{d}\fan61
 \thnline\dl0121\dl2130\dl2132\dl3041\dl3241\dl4161\dl3032\hskip6\unit
\]
The $2$-limited set is $X\cup\{u,v\}$ where $X$ is $2$-limited in $G\setminus\{a,b,c,d,u,v\}$.
This applies even if $a,b,c,d$ are not all distinct.

Now suppose there exists a d-edge $uv$ that lies in a single triangle. Consider the transformation
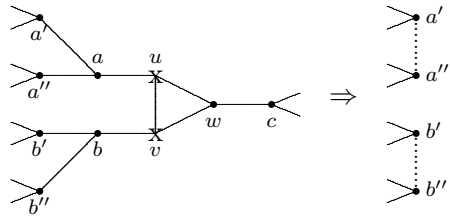
\[
 \naf03\naf02\naf01\naf00\ptd03{a'}\ptd02{a''}\ptd01{b'}\ptd00{b''}
 \thnline\dl0011\dl0111\dl0212\dl0312\ptu12{a}\ptd11{b}
 \dl1222\dl1121\pxu22{u}\pxd21{v}\dl2122\dl213{1.5}\dl223{1.5}\ptd{3}{1.5}{w}
 \dl3{1.5}4{1.5}\fan4{1.5}\ptd{4}{1.5}{c}
 \hskip4\unit\qquad\raise1.5\unit\hbox{$\Rightarrow$}\qquad
 \naf03\naf02\naf01\naf00\ptr03{a'}\ptr02{a''}\ptr01{b'}\ptr00{b''}
 \dotline\dl0001\dl0203
\]
We consider the set $X\cup\{u,v\}$ where $X$ is $2$-limited in $G\setminus\{a,b,c,u,v,w\}$
with the c-edges $a'a''$, $b'b''$ added if necessary. Once again, a number of degenerate cases
are covered here, although we know that $a\ne b$ as $uv$ is not in two triangles.
However, we do need to consider the subcase where the addition
of both $a'a''$ and $b'b''$ are required and together give rise to a $K_4$ in the c-edges.
Note that in this case all edges incident to $a$ or $b$ are d-edges.
We consider $X\cup\{a',a'',b\}$ where $X$ is $2$-limited in $G$ with the
$K_4$ and $\{a,b,u,v,w\}$ removed. Note that this is $2$-limited: $a'a''$ is not an edge
of $G$ and $\{b',b''\}\ne\{a',a''\}$ so $b$ is adjacent to at most one of $a',a''$.
 
Finally, consider a d-edge $uv$ that does not lie in a triangle. Consider the
following transformation.
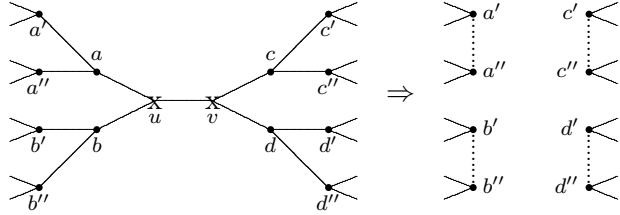
\[
 \naf03\naf02\naf01\naf00\ptd03{a'}\ptd02{a''}\ptd01{b'}\ptd00{b''}
 \thnline\dl0011\dl0111\dl0212\dl0312\ptu12{a}\ptd11{b}
 \dl122{1.5}\dl112{1.5}\pxd2{1.5}{u}\dl2{1.5}3{1.5}\pxd3{1.5}{v}
 \dl423{1.5}\dl413{1.5}\ptu42{c}\ptd41{d}\dl4150\dl4151\dl4252\dl4253
 \ptd53{c'}\ptd52{c''}\ptd51{d'}\ptd50{d''}\fan50\fan51\fan52\fan53
 \hskip5\unit\qquad\raise1.5\unit\hbox{$\Rightarrow$}\qquad
 \naf03\naf02\naf01\naf00\ptr03{a'}\ptr02{a''}\ptr01{b'}\ptr00{b''}
 \ptl23{c'}\ptl22{c''}\ptl21{d'}\ptl20{d''}\fan20\fan21\fan22\fan23
 \dotline\dl0001\dl0203\dl2021\dl2223
 \hskip2\unit
\]
We consider $X\cup\{u,v\}$ where $X$ is $2$-limited in $G\setminus\{a,b,c,d,u,v\}$
with c-edges $a'a''$, $b'b''$, $c'c''$, and $d'd''$ added if necessary.
Note that $a,b,c,d$ are distinct as $uv$ does not lie in a triangle.
Also if $a'a''$, say, needs to be added then all edges incident to $a$ are
d-edges. As there must be some d-edge somewhere, we are done except in
the cases when one or more c-edge $K_4$s are formed by the addition of $\{a'a'',b'b'',c'c'',d'd''\}$.

No single added c-edge can form a c-edge $K_4$ due to the absence of configuration~$A$.
Suppose first that adding just the c-edges $a'a''$ and $b'b''$ forms a c-edge $K_4$.
We then consider $X\cup\{a',a'',b\}$ where $X$ is $2$-limited for $G$ with the $K_4$
and $\{a,b,u,v\}$ removed. If adding $a'a''$ and $c'c''$ forms a c-edge $K_4$,
consider $X\cup\{a',a'',c\}$ where $X$ is $2$-limited for $G$ with the $K_4$
and $\{a,c,u,v\}$ removed. Now suppose we need three added c-edges to form a
c-edge $K_4$, say $a'a''$, $b'b''$, $c'c''$. Then we take $X\cup\{a',a'',b,v\}$
where $X$ is $2$-limited in $G$ with the $K_4$ and $\{a,b,c,d,u,v\}$ removed and with
the c-edge $d'd''$ added if necessary. Finally, if all four added c-edges
are needed to form a c-edge $K_4$ then $G$ contains just 10 vertices
and we can take $\{a,b,c,d\}$ as our $2$-limited set. 
\end{proof}

For an arbitrary graph $G$ of maximum degree $3$, let all edges be $d$-edges and then since there are no $c$-edges, Theorem \ref{thm:cubic-bd} applies directly.

\begin{corollary}
For any graph $G$ with $\Delta(G) = 3$,
\[
L_2(G) \geq \frac{|V(G)|}{3}.
\]
\end{corollary}

Recall that if $G$ is an $r$-regular graph and $X$ is a $k$-limited packing in $G$, then $V (G) \setminus X$ is a $(r + 1 - k)$-tuple dominating set. The size of the smallest $\ell$-tuple dominating set in a graph $G$ is denoted by $\gamma_{\times \ell}(G)$. As was noted by Gallant, Gunther, Hartnell and Rall \cite{GGHR10}, if $G$ is an $r$-regular graph and $k \geq 1$, then
\[
L_k(G) + \gamma_{\times (r- k+1)}(G) = |V (G)|.
\]
Harant and Henning \cite{HH05} showed that for any graph $G$, with $n$ vertices, minimum degree $\delta$ and average degree $d$, then $\gamma_{\times 2} (G) \leq \frac{(\ln(1+d) + \ln \delta + 1)n}{\delta}$.  This bound was improved by Cockayne and Thomason \cite{CT08} who showed that $\gamma_{\times 2}(G) \leq  \frac{(\ln(1+\delta) + \ln \delta + 1)n}{\delta}$.  Neither of these bounds are of use when $\delta = d = 3$.  Theorem \ref{thm:cubic-bd} shows that if $G$ is a $3$-regular graph on $n$ vertices, then
\[
\gamma_{\times 2}(G) \leq n - L_2(G) \leq \frac{2n}{3}. 
\]

\section{Graphs with large degree}\label{sec:large-deg}

When $k \geq 2$ is fixed and $\Delta$ tends to infinity, the lower bound for $L_k(G)$ given by Gagarin and Zverovich \cite{GZ13}, as in equation \eqref{eq:random-del-bd}, shows that if $G$ if a graph with $\Delta(G) = \Delta$, then
\begin{equation}\label{eq:k-fixed-bd}
L_k(G) > n \frac{k}{k+1}\left(\frac{1}{(k+1)\binom{\Delta+1}{k+1}}\right)^{1/k} \geq \frac{nk}{e \Delta^{1+1/k}}.
\end{equation}

The following example shows that up to the constant factor $e^{-1}$, this is best possible.

\begin{example}
For any $k \geq 1$ and $q = p^n$, a prime power, define a graph $G_{q,k}$ whose vertex set is the points of a $k + 1$ dimensional projective space over $GF(q)$. That is, vertices are equivalence classes of non-zero elements of $GF(q)^{k+2}$, where a pair of vectors in $GF(q)^{k+2}$ are equivalent if one is a non-zero multiple of the other.  Join two vertices by an edge if their inner product is $0$ (in $GF(q)$).

In this graph, any $k$ vertices trivially form a $k$-limited packing.  On the other hand, any collection of $k+1$ points lie on some $k$-dimensional hyperplane.  The normal vector to this hyperplane is adjacent to all points on the hyperplane and hence to all $k+1$ points.  Thus, the largest $k$-limited packing in this graph is of size $k$.

The number of vertices in $G_{q,k}$ is $\frac{q^{k+2}-1}{q-1}$ and every vertex has $\frac{q^{k+1}-1}{q-1}$ neighbours.  Thus, as $q$ tends to infinity,
\[
L_k(G_{q,k}) = k = \frac{k |V(G_{q,k})|}{\Delta(G_{q,k})^{1+1/k}} (1+o(1)).
\]
\end{example}

In some cases where both $k$ and $\Delta$ are tending to infinity, a better bound than that given by equation \eqref{eq:k-fixed-bd} can be obtained using the Lov\'{a}sz Local Lemma.  The `symmetric' version of the Lov\'{a}sz Local Lemma \cite{EL75}, which can also be found in \cite{bB01}, gives the following result about events in a probability space without too much dependence.  Let $B_1, B_2, \ldots, B_n$ be events in a probability space and let $d \geq 1$ be such that for every $i \leq n$, there is a set $I_i$ with $|I_i| \leq d$ so that $B_i$ is mutually independent of $\{B_j \mid j \notin I_i \cup \{i\}\}$.  If for each $i \leq n$, we have $\mathbb{P}(B_i) \leq \frac{1}{e(d+1)}$, then
\[
\mathbb{P}\left(\cap_{i=1}^n \bar{B}_i \right) \geq \left(1 - \frac{1}{d} \right)^n.
\]

\begin{theorem}\label{thm:lll-bd}
Let $k > \log \Delta \log\log \Delta$.  For any graph $G$ with $\Delta(G) \leq \Delta$,
\[
L_k(G) \geq \frac{k |V(G)|}{\Delta}(1+o(1)),
\]
as $\Delta \to \infty$.
\end{theorem}

\begin{proof}
Fix a graph $G$ on $n$ vertices with $\Delta(G) = \Delta$ and let $k \geq \log \Delta \log\log \Delta$.  Set $\varepsilon_1 = (5/\log\log\Delta)^{1/2}$ and $p = (1-\varepsilon_1) (k+1)/(\Delta+1)$.  Choose a set of vertices, $X$, in $G$ independently at random, each vertex chosen with probability $p$.  For each $v \in V(G)$, let $B_v$ be the event that $|N[v] \cap X| \geq k+1$.  Then, the event $\cap \bar{B}_v$ is the event that $X$ is a $k$-limited packing.

Using the Chernoff bound for binomial random variables,
\begin{align*}
\mathbb{P}(B_v)
	& = \mathbb{P}\left(\operatorname{Bin}(\deg(v) + 1, p) \geq k+1\right)\\
	& \leq \mathbb{P} \left(\operatorname{Bin}(\Delta+1, p) \geq k+1\right)\\
	& \leq \exp\left(- \frac{(k+1 - p(\Delta+1))^2}{2p(\Delta+1)+ \frac{2}{3}(k+1-p(\Delta+1))} \right)\\
	& = \exp\left(- \frac{\varepsilon_1^2 (k+1)}{2 - 4\varepsilon_1/3} \right)\\
	&\leq \exp\left( - \frac{\varepsilon_1^2(k+1)}{2}\right)\\
	& = \exp\left(- \frac{5(\log\Delta\log\log\Delta + 1)}{2 \log\log\Delta}\right) < \exp(-1-2\log(\Delta+1)).
\end{align*}

If vertices $v$ and $w$ are such that $d(v, w) \geq 3$, then $B_v$ is independent of $B_w$.  Thus, each event $B_v$ is mutually independent of all but at most $\Delta + \Delta(\Delta-1) = \Delta^2$ events of the form $B_w$.  Since for every $v \in V(G)$,
\[
\mathbb{P}(B_v) < \exp(-1-2\log(\Delta+1))  = \frac{1}{e(\Delta+1)^2} \leq \frac{1}{e(\Delta^2  +1)},
\]
then by the Local Lemma, the probability that $X$ is a $k$-limited packing is at least, when $\Delta$ is large enough,
\begin{equation}\label{eq:lll-k-lim}
\mathbb{P}(\cap \bar{B}_v) \geq \left(1 - \frac{1}{\Delta^2}\right)^n > \exp\left(\frac{-2n}{\Delta^2} \right).
\end{equation}

On the other hand, consider the probability that the set $X$ is much smaller than its expected size.  Set $\varepsilon_2 = 3/\sqrt{k\Delta}$.  Then, again by the Chernoff bound, and using the fact that $p = (1-\varepsilon_1)(k+1)/(\Delta+1) > k/(2\Delta)$,
\begin{align*}
\mathbb{P}(|X| < (1-\varepsilon_2) np)
	& \leq \exp\left(- \frac{(np - (1-\varepsilon_2)np)^2}{2np} \right)\\
	&  = \exp\left(-\frac{\varepsilon_2^2 np}{2}\right)\\
	& \leq \exp\left(- \frac{\varepsilon_2^2 n k}{4 \Delta} \right)\\
	& = \exp\left(- \frac{9 n}{4 \Delta^2} \right)\\
	& < \exp\left(-\frac{2n}{\Delta^2}\right).
\end{align*}

Thus, $\mathbb{P}(|X| < (1-\varepsilon_2)np) < \mathbb{P}(\cap \bar{B}_v)$ and so there is at least one choice of a set $X$ so that $X$ is a $k$-limited packing and $|X| \geq (1-\varepsilon_2) np$.  Using the fact that $k \geq \log \Delta \log\log \Delta$ and hence $\varepsilon_2 <  \varepsilon_1/3$, then
\[
L_k(G) \geq \frac{(k+1) n}{(\Delta+1)} \left(1 - \frac{3}{\sqrt{\log\log\Delta}} \right) = \frac{k n}{\Delta}(1+o(1)).
\]
which completes the proof.
\end{proof}

Note that for any graph $G$ with minimum degree $\delta(G)$, by double counting,
\[
L_k(G) \leq \frac{k |V(G)|}{\delta(G) +1}
\]
and, in particular, for $\Delta$-regular graphs, Theorem \ref{thm:lll-bd} is asymptotically best possible, for all values of $k$ for which the result applies.

\end{document}